\documentclass[10pt, A4, leqno]{amsart}

\usepackage{enumerate}
\usepackage{graphicx} 
\usepackage{tikz-cd}
\usepackage{physics}
\usepackage{hyperref}

\makeatletter
\@namedef{subjclassname@2020}{\textup{2020} Mathematics Subject Classification}
\makeatletter

\numberwithin{equation}{section}

\theoremstyle{plain}
  \newtheorem{theorem}{Theorem}[section]
  \newtheorem*{theorem*}{Theorem}
  \newtheorem{proposition}[theorem]{Proposition}
  \newtheorem{fact}[theorem]{Fact}
  \newtheorem*{fact*}{Fact}

\theoremstyle{definition}
  \newtheorem{definition}[theorem]{Definition}
\theoremstyle{remark}
  \newtheorem{remark}{Remark}[section]
  \newtheorem*{acknowledgements}{Acknowledgements}
  \newtheorem{example}[theorem]{Example}
\numberwithin{equation}{section}

\newcommand{\R}{{\mathbb{R}}}

\renewcommand{\phi}{\varphi}
\renewcommand{\epsilon}{\varepsilon}

\newcommand{\inner}[2]{\left\langle{#1},{#2}\right\rangle}

\DeclareMathOperator{\Vol}{Vol}
\DeclareMathOperator\Ric{Ric}
\DeclareMathOperator\II{II}

\newcommand{\setm}{\,;\,}
\newcommand{\Nabla}{\overline{\nabla}}
\newcommand{\Ell}{\vb*{\ell}}
\newcommand{\vbp}{\vb*{p}}
\newcommand{\vbq}{\vb*{q}}
\newcommand{\vbH}{\vb*{H}}

\allowdisplaybreaks

\title[]{The volume of marginally trapped submanifolds in a Lorentzian manifold satisfying \\ the null energy condition}
\author[]{Riku Kishida}
\address[Riku Kishida]{Department of Mathematical and Computing Sciences, Institute of Science Tokyo, Tokyo 152-8552, Japan}
\email{kishida.r.1632@m.isct.ac.jp}

\date{June 8, 2025.}

\begin{document}

\begin{abstract}
  In this paper, we focus on a marginally trapped submanifold $f:\Sigma^n\to M^{n+2}_1$ in a Lorentzian manifold $M^{n+2}_1$.
  We show that $f$ lies in a certain null hypersurface $\mathcal{N}_f^{n+1}$ in $M^{n+2}_1$ and $f$ has a locally volume-maximizing property in $\mathcal{N}^{n+1}_f$ if $M^{n+2}_1$ satisfies the null energy condition.
\end{abstract}

\maketitle

\section*{Introduction}

Let $\Sigma^n$ be an $n$-dimensional smooth manifold, and let $\Lambda^{n+1}$ be the $(n+1)$-dimensional light-cone in the $(n+2)$-dimensional Minkowski spacetime $\R^{n+2}_1$ with signature $(-+\dots+)$.
In author's previous work \cite{Kishida24}, the first and the second variational formulas for a space-like immersion $\vbp:\Sigma^n\to\Lambda^{n+1}$ were given, and the null hypersurface $\mathcal{N}^{n+1}_{\vbp}$, which is called the \textit{null-space} with respect to $\vbp$, was constructed by using the dual map $\vbq:\Sigma^n\to\Lambda^{n+1}$ of $\vbp$.
Moreover, the followings were shown in \cite{Kishida24}.
\begin{itemize}
  \item $\vbp$ is a critical point of the volume functional for $\vbq$-direction variations if and only if the scalar curvature with respect to the induced metric by $\vbp$ is identically zero.
  \item In that case, $\vbp$ has a locally volume-maximizing property in the null-space $\mathcal{N}^{n+1}_{\vbp}$ (more precisely, for any variation of $\vbp$ in $\mathcal{N}_{\vbp}^{n+1}$ the first variation is zero and the second variation is non-positive).
\end{itemize}
In this paper, we give a significant generalization of these results.

Let $M^{n+2}_1$ be an $(n+2)$-dimensional Lorentzian manifold, and we denote by $\inner{\cdot}{\cdot}$ the Lorentzian metric of $M^{n+2}_1$.
We say that a space-like immersion $f:\Sigma^n\to M^{n+2}_1$ is a \textit{marginally trapped submanifold} if there exist a light-like normal vector filed $\Ell_+$ satisfying $\inner{\Ell_+}{\vbH}$ is identically zero, where $\vbH$ is the mean curvature vector field of $f$.
A space-like immersion $\vbp:\Sigma^n\to\Lambda^{n+1}$ with vanishing scalar curvature is a typical example of marginally trapped submanifolds in $\R^{n+2}_1$ (cf. Example~\ref{exam_1}).
If $n=2$, such an immersion $f$ is called a \textit{marginally trapped surface}.
In Honda-Izumiya \cite{HI15}, many examples of marginally trapped surfaces are given for the case of $M^4_1:=\R^4_1$.
Moreover, there are many examples of marginally trapped submanifolds in general dimensions (cf. Example~\ref{exam_2}).

For a marginally trapped surface in a $4$-dimensional Lorentzian manifold $M^4_1$, Andersson-Metzger \cite{AM10} gives a formula for the derivative of $\vbH$ with respect to a given variational vector field.
On the other hand, Honda-Izumiya \cite{HI15} gives the first and the second variational formulas for the volume of a marginally trapped surface in $\R^4_1$.

Regarding these results,
we begin to compute the first and the second variational formulas for a marginally trapped submanifold $f:\Sigma^n\to M^{n+2}_1$ without any restriction of the dimension $n$.
The goal of this paper is to show that the results of the author's previous work still hold true even if the following substitutions are made:
\begin{itemize}
  \item The ambient space $\R^{n+2}_1$ is an example of a Lorentzian manifold $M^{n+2}_1$ satisfying the null energy condition. So we replace $\R^{n+2}_1$ with $M^{n+2}_1$.
  \item A space-like immersion $f:\Sigma^n\to\Lambda^{n+1}$ with vanishing the scalar curvature is an example of a marginally trapped submanifold in $\R^{n+2}_1$. So we replace $f$ with a marginally trapped submanifold in $M^{n+2}_1$.
\end{itemize}
With this substitution, we construct the null-space $\mathcal{N}_f^{n+1}$ for a marginally trapped submanifold $f:\Sigma^n\to M^{n+2}_1$.
The main result of this paper is as follows:

\begin{theorem*}[A generalization of {\cite[Theorem in Introduction]{Kishida24}}]
  Let $M^{n+2}_1$ be an $(n+2)$-dimensional Lorentzian manifold satisfying the null energy condition {\rm (}called NEC, that is, $M^{n+2}_1$ is light-like geodesically complete and the Ricci tensor of $M^{n+2}_1$ is non-positive along light-like directions{\rm )}, and let $f:\Sigma^n\to M^{n+2}_1$ be a marginally trapped submanifold.
  Then, for any compactly supported variation of $f$ in the null-space $\mathcal{N}_f^{n+1}$, the first variation on the volume is zero and the second variation is non-positive.
  Therefore, $f$ is a locally volume-maximizing hypersurface in $\mathcal{N}_f^{n+1}$.
\end{theorem*}

The paper is organized as follows:
In Section~\ref{sec_1}, we explain basic facts of marginally trapped submanifolds and null energy condition.
In Section~\ref{sec_2}, we give the second variational formula on the volume of a marginally trapped submanifold.
In Section~\ref{sec_3}, we give the definition of the null-space for a marginally trapped submanifold and prove the main theorem.

\section{preliminaries}\label{sec_1}

In this paper, we assume that all mappings and manifolds are $C^\infty$-differentiable.
In Subsection~\ref{subsec_1}, we introduce basic notations on marginally trapped submanifolds.
In Subsection~\ref{subsec_2}, we explain properties of Lorentzian manifolds satisfying the null energy condition.

\subsection{Marginally trapped submanifold}\label{subsec_1}

We fix an $(n+2)$-dimensional Lorentzian manifold $M^{n+2}_1$ $(n\ge 2)$, and we denote by $\inner{\,}{\,}$ the Lorentzian metric on $M^{n+2}_1$.
Let $\Sigma^n$ be an $n$-dimensional manifold and $f:\Sigma^n\to M^{n+2}_1$ a space-like immersion.
We denote by $T\Sigma^n$ and $N\Sigma^n$ the tangent bundle and the normal bundle of $\Sigma^n$ in $M^{n+2}_1$, respectively.
In addition, we denote by $\mathfrak{X}(\Sigma^n)$ and $\Gamma(N\Sigma^n)$ the sets of global sections of $T\Sigma^n$ and $N\Sigma^n$, respectively.
We define the second fundamental form $\II:\mathfrak{X}(\Sigma^n)\times\mathfrak{X}(\Sigma^n)\to\Gamma(N\Sigma^n)$ with respect to $f$ as
\[\II(X,Y):=(\overline\nabla_XY)^\bot\qquad(X,Y\in\mathfrak{X}(\Sigma^n)),\]
where $\Nabla$ is the Levi-Civita connection of $M^{n+2}_1$, and $\bot$ represents the projection into the normal bundle $N\Sigma^n$.
We also define the \textit{mean curvature vector field} $\vbH\in\Gamma(N\Sigma^n)$ with respect to $f$ in $M^{n+2}_1$ as
\[\vbH:=\sum_{i=1}^n\II(e_i,e_i),\]
where $\{e_1,\dots,e_n\}$ is a local orthonormal frame field of $T\Sigma^n$.
The definition of marginally trapped submanifolds is given as follows:

\begin{definition}[cf. \cite{AM10}, \cite{HI15}]
  Let $f:\Sigma^n\to M^{n+2}_1$ be a space-like immersion of codimension 2.
  If there exist a global section $\Ell_+\in\Gamma(N\Sigma^n)$ such that $\Ell_+$ is non-zero at each point of $\Sigma^n$ and satisfies $\inner{\Ell_+}{\Ell_+}=0$ and $\inner{\Ell_+}{\vbH}=0$, we say that $f$ is a \textit{marginally {\rm (}outer{\rm )} trapped submanifold} (with respect to $\Ell_+$).
\end{definition}

Given a marginally trapped submanifold $f:\Sigma^n\to M^{n+2}_1$, we can take two global sections $\Ell_+,\;\Ell_-\in\Gamma(N\Sigma^n)$ satisfying
\begin{align}\label{eq_Ell}
  \inner{\Ell_+}{\Ell_+}=\inner{\Ell_-}{\Ell_-}=0,\qquad\inner{\Ell_+}{\Ell_-}=-2,\qquad\inner{\Ell_+}{\vbH}=0.
\end{align}
By defining
\[\theta^+:=\inner{\vbH}{\Ell_+},\qquad\theta^-:=\inner{\vbH}{\Ell_-},\]
we can express $\vbH=-(\theta^-\Ell_++\theta^+\Ell_-)/2$.
The definition of a marginally trapped submanifold implies that $\theta^+$ is identically zero, that is, $\vbH=-\;\theta^-\Ell_+/2$ holds.

\begin{example}\label{exam_1}
  Suppose that $M^{n+2}_1$ is the $(n+2)$-dimensional Minkowski spacetime $\R^{n+2}_1$ with signature $(-+\dots+)$, and we denote by $\Lambda^{n+1}$ the $(n+1)$-dimensional light-cone, that is, we define
  \[\Lambda^{n+1}:=\left\{\vb*{x}\in\R^{n+2}_1\setm\langle\vb*{x},\vb*{x}\rangle=0\right\}.\]
  Let $\vbp:\Sigma^n\to\Lambda^{n+1}$ be a space-like immersion satisfying that the scalar curvature with respect to the induced metric by $\vbp$ is identically zero.
  We define $\vbq:\Sigma^n\to\Lambda^{n+1}$ as the dual map of $\vbp$, that is, $\vbq$ is the unique map satisfying
  \[\inner{\vbq}{\vbq}=0\qquad\inner{\vbp}{\vbq}=1,\qquad\inner{X}{\vbq}=0\qquad(X\in\mathfrak{X}(\Sigma^n)).\]
  Then, the mean curvature vector field $\vbH$ of $\vbp$ can be written as $\vbH=-n\vbq$ (cf. \cite[Propositions~1.2 and 1.3]{Kishida24}).
  This implies that $\vbp$ is a marginally trapped submanifold in $\R^{n+2}_1$ by setting $\Ell_+:=\vbq$ and $\Ell_-:=-2\vbp$.
\end{example}

\begin{example}\label{exam_2}
  In the case of $M^{n+2}_1:=\R^{n+2}_1$, other typical examples of marginally trapped submanifolds are as follows:
  \begin{enumerate}[(a)]
    \item Let $f:\Sigma^n\to\R^{n+1}$ be a minimal hypersurface in the Euclidean space and $\nu$ its unit normal vector field, then the map $\Sigma^n\ni x\mapsto (0,f(x))\in\R^{n+2}_1$ is a marginally trapped submanifold by setting $\Ell_+:=(1,\nu)$ and $\Ell_-:=(1,-\nu)$ (we can also set $\Ell_+:=(1,-\nu),\;\Ell_-:=(1,\nu)$).
    \item Let $f:\Sigma^n\to\R^{n+1}_1$ be a space-like zero mean curvature hypersurface and $\nu$ its unit normal vector field, then the map $\Sigma^n\ni x\mapsto (f(x),0)\in\R^{n+2}_1$ is a marginally trapped submanifold by setting $\Ell_+:=(\nu,1)$ and $\Ell_-:=(\nu,-1)$  (we can also set $\Ell_+:=(\nu,-1),\;\Ell_-:=(\nu,1)$).
    \item We denote by $\mathbb{H}^{n+1}$ the $(n+1)$-dimensional hyperbolic space whose sectional curvature is equal to $-1$, that is, we define
    \[\mathbb{H}^{n+1}:=\{\vb*{x}=(x^1,\dots,x^{n+2})\in\R^{n+2}_1\setm\inner{\vb*{x}}{\vb*{x}}=-1,\;x^1>0\}.\]
    Let $f:\Sigma^n\to\mathbb{H}^{n+1}$ be a space-like hypersurface and $\nu$ its unit normal vector field, and we assume that the mean curvature vector field of $f$ coincides with $n\nu$ (i.e. $f$ is a space-like hypersurface in $\mathbb{H}^{n+1}$ whose mean curvature function is equal to 1).
    Then, the map $\Sigma^n\ni x\mapsto f(x)\in\R^{n+2}_1$ is a marginally trapped submanifold by setting $\Ell_+:=\nu+f$ and $\Ell_-:=-\nu+f$.
    In this case, the mean curvature vector field $\vbH$ can be written as $\vbH=n\Ell_+$.
    \item We denote by $\mathbb{S}_1^{n+1}$ the $(n+1)$-dimensional de Sitter spacetime whose sectional curvature is equal to 1, that is, we define
    \[\mathbb{S}^{n+1}_1:=\{\vb*{x}\in\R^{n+2}_1\setm\inner{\vb*{x}}{\vb*{x}}=1\}.\]
    Let $f:\Sigma^n\to\mathbb{S}_1^{n+1}$ be a space-like hypersurface and $\nu$ its unit normal vector field, and we assume that the mean curvature vector field of $f$ coincides with $-n\nu$ (i.e. $f$ is a space-like hypersurface in $\mathbb{S}_1^{n+1}$ whose mean curvature function is equal to 1).
    Then, the map $\Sigma^n\ni x\mapsto f(x)\in\R^{n+2}_1$ is a marginally trapped submanifold by setting $\Ell_+:=-\nu-f$ and $\Ell_-:=-\nu+f$.
    In this case, the mean curvature vector field $\vbH$ can be written as $\vbH=n\Ell_+$.
  \end{enumerate}
\end{example}

For a marginally trapped submanifold $f:\Sigma^n\to M^{n+2}_1$, we define two shape operators $A_\pm:\mathfrak{X}(\Sigma^n)\to\mathfrak{X}(\Sigma^n)$ so that
\[\inner{A_\pm X}{Y}:=\inner{\II(X,Y)}{\Ell_\pm}=-\inner{\Nabla_X\Ell_\pm}{Y}\qquad(X,Y\in\mathfrak{X}(\Sigma^n)).\]
Using the traces of the shape operators $A_\pm$, two real valued functions $\theta^\pm$ can be expressed as follows:

\begin{proposition}\label{prop_shape_op}
  We have $\theta^\pm=\tr(A_\pm)$, where $\tr(A_\pm)$ are the traces of $A_\pm$.
\end{proposition}
\begin{proof}
  Let $\{e_1,\dots,e_n\}$ be a local orthonormal frame field of $T\Sigma^n$.
  Then, we have
  \[\tr(A_\pm)=\sum_{i=1}^n\inner{A_\pm e_i}{e_i}=\sum_{i=1}^n\inner{\II(e_i,e_i)}{\Ell_\pm}=\inner{\vbH}{\Ell_\pm}=\theta^\pm.\]
\end{proof}

Therefore, for a marginally trapped submanifold, the trace of $A_+$ is zero.

\subsection{Null energy condition}\label{subsec_2}

We denote by $R$ the curvature tensor of $M^{n+2}_1$, which is defined as
\begin{align*}
R(X,Y)Z&:=\Nabla_X\Nabla_YZ-\Nabla_Y\Nabla_XZ-\Nabla_{[X,Y]}Z,\\
R(X,Y,Z,W)&:=\inner{R(X,Y)Z}{W}\qquad(X,Y,Z,W\in\mathfrak{X}(M^{n+2}_1)),
\end{align*}
where $[\cdot,\cdot]$ is the Lie bracket of vector fields.
In addition, we denote by $\Ric$ the \textit{Ricci tensor} of $M^{n+2}_1$, which is defined as
\[\Ric(X,Y):=-R(X,e_1,e_1,Y)+\sum_{i=2}^{n+2}R(X,e_i,e_i,Y)\qquad(X,Y\in\mathfrak{X}(M^{n+2}_1)),\]
where $\{e_1,\dots,e_{n+2}\}$ is a local orthonormal frame field of $TM^{n+2}_1$ satisfying
\[\inner{e_1}{e_1}=-1,\qquad\inner{e_i}{e_i}=1\qquad(i=2,\dots,n+2).\]

\begin{definition}
  Let $M^{n+2}_1$ be an $(n+2)$-dimensional Lorentzian manifold.
  We call that $M^{n+2}_1$ satisfies the \textit{null energy condition} (called \textit{NEC}) if $M^{n+2}_1$ is light-like geodesically complete and $\Ric(X,X)\ge 0$ holds for any $X\in\mathfrak{X}(M^{n+2}_1)$ pointing in light-like directions.
\end{definition}

For example, light-like geodesically complete Einstein manifolds are typical examples of Lorentzian manifolds satisfying the null energy condition, and in particular $\R^{n+2}_1$ satisfies the null energy condition.
As a remarkable property of a light-like hypersurface in a Lorentzian manifold satisfying the null energy condition, the following fact is known.

\begin{fact}[Akamine-Honda-Umehara-Yamada~\cite{AHUY20}]\label{fact_ahuy}
  Let $M^{n+2}_1$ be an $(n+2)$-dimensional Lorentzian manifold satisfying the null energy condition, and let $N^{n+1}$ be an $(n+1)$-dimensional manifold.
  Suppose that $\Phi:N^{n+1}\to M^{n+2}_1$ is a proper map and a light-like immersion {\rm (}i.e. the induced metric by $\Phi$ is positive semi-definite but not positive definite at each point of $N^{n+1}${\rm )}.
  Then, $\Phi$ is totally geodesic.
\end{fact}

\section{The variational formulas of marginally trapped submanifold}\label{sec_2}

In this section, we calculate variational formulas on the volume of marginally trapped submanifold with respect to a light-like direction.

We assume that $\Sigma^n$ is a compact oriented manifold with boundary.
Let $F:(-\epsilon,\epsilon)\times\Sigma^n\to M^{n+2}_1$ be a variation of a marginally trapped submanifold $f:\Sigma^n\to M^{n+2}_1$ with fixed boundary.
For the variation $F$, we define the vector field $X$ along $f$ as
\[X:=\left.\frac{\partial F}{\partial t}\right|_{t=0}\]
called the \textit{variational vector field} with respect to $F$. 
Moreover, we define the vector field $X'$ along $f$ as
\[X':=\left.\left(\Nabla_{\partial/\partial t}\frac{\partial F}{\partial t}\right)\right|_{t=0}.\]
In \cite{Kishida24}, we introduced ``admissible" and ``characteristic" variations for space-like hypersurfaces in the light-cone as well behaved variations with respect to a light-like direction, which can be extended to variations of a marginally trapped submanifolds as follows:

\begin{definition}[A generalization of {\cite[Definitions~2.1 and 2.6]{Kishida24}}]
  Let $f:\Sigma^n\to M^{n+2}_1$ be a marginally trapped submanifold with respect to $\Ell_+$, and let $F:(-\epsilon,\epsilon)\times\Sigma^n\to M^{n+2}_1$ be a variation of $f$ with fixed boundary.
  Then, $F$ is called an \textit{admissible variation} of $f$ if there exists a real-valued function $\phi$ defined on $\Sigma^n$ which satisfies $X=\phi\Ell_+$.
  Moreover, $F$ is called a \textit{characteristic variation} of $f$ if $F$ is admissible and there exists a real-valued function $\psi$ defined on $\Sigma^n$ which satisfies $X'=\psi\Ell_+$.
\end{definition}

We denote by $dV_t$ the Riemannian volume form of the induced metric determined by $F(t,\cdot)$, and we set $\Vol(t):=\int_{\Sigma^n}dV_t$.
It is well-known that
\[\left.\frac{d}{dt}\Vol(t)\right|_{t=0}=-\int_{\Sigma^n}\inner{X}{\vbH}dV_0\]
holds, and the following proposition follows immediately.

\begin{proposition}\label{prop_first_variation}
  Let $\Sigma^n$ be a compact oriented manifold with boundary, and let $f:\Sigma^n\to M^{n+2}_1$ be a marginally trapped submanifold.
  Then, for any admissible variation with fixed boundary, the first variation on the volume is zero.
\end{proposition}

It is known that the second variation on the volume in the general case can be expressed as described in the following proposition (see \cite[Section 6]{SW01} and \cite[Appendix A]{Kishida24} for the proof of the proposition).
\begin{proposition}\label{prop_second_variation_general}
  Let $(M^{n+2}_1,\inner{\cdot}{\cdot})$ be an $(n+2)$-dimensional Lorentzian manifold and $\Sigma^n$ a compact oriented manifold with boundary.
  For a variation $F$ of a space-like immersion $f:\Sigma^n\to M^{n+2}_1$ with fixed boundary, the second variation on the volume is given by
  \begin{equation}\begin{split}\label{eq_second_variation_general}
    \left.\frac{d^2}{dt^2}\Vol(t)\right|_{t=0}
    =&\int_{\Sigma^n}\Biggl(
    -\sum_{i=1}^nR(X,e_i,e_i,X)
    +\sum_{i=1}^n\inner{(\Nabla_{e_i}X)^\bot}{(\Nabla_{e_i}X)^\bot}\\
    &\qquad\qquad
    -\sum_{i,j=1}^n\inner{\Nabla_{e_i}X}{e_j}\inner{\Nabla_{e_j}X}{e_i}\\
    &\qquad\qquad
    +\sum_{i,j=1}^n\inner{\Nabla_{e_i}X}{e_i}\inner{\Nabla_{e_j}X}{e_j}
    -\inner{X'}{\vbH}
  \Biggr)dV_0,
  \end{split}\end{equation}
  where $\{e_1,\dots,e_n\}$ is a local orthonormal frame field of $T\Sigma^n$.
\end{proposition}

Using this formula, we can show the following.

\begin{proposition}[A generalization of {\cite[Proposition~2.8]{Kishida24}}]\label{prop_volume_var_nec}
  Let $f:\Sigma^n\to M^{n+2}_1$ be a marginally trapped submanifold with respect to $\Ell_+$, and let $F:(-\epsilon,\epsilon)\times\Sigma^n\to M^{n+2}_1$ be a characteristic variation of $f$ with fixed boundary.
  Suppose that the variational vector field $X$ of $F$ is written as $X=\phi\Ell_+$ by a real-valued function $\phi$ on $\Sigma^n$.
  Then, the second variation on the volume is given by
  \begin{align}\label{eq_volume_var_nec}
    \left.\frac{d^2}{dt^2}\Vol(t)\right|_{t=0}=-\int_{\Sigma^n}\phi^2\Bigl(\tr(A_+^2)+\Ric(\Ell_+,\Ell_+)\Bigr)dV_0,
  \end{align}
  where $A_+^2$ is the composition of two $A_+$.
\end{proposition}
\begin{proof}
  We calculate each term in the integrand function of \eqref{eq_second_variation_general}.
  First, we consider the term relating to the curvature tensor.
  Let $\{e_1,\dots,e_n\}$ be a local orthonormal frame field of $T\Sigma^n$.
  By taking $\Ell_\pm\in\Gamma(N\Sigma^n)$ satisfying \eqref{eq_Ell}, we have the local orthonormal frame field
  \[\left\{\frac{\Ell_++\Ell_-}{2},\;\frac{\Ell_+-\Ell_-}{2},\;e_1,\dots,e_n\right\}\]
  of $TM^{n+2}_1$.
  Then, by using the symmetry of the curvature tensor $R$, we can easily check that
  \[R\left(\Ell_+,\frac{\Ell_++\Ell_-}{2},\frac{\Ell_++\Ell_-}{2},\Ell_+\right)=R\left(\Ell_+,\frac{\Ell_+-\Ell_-}{2},\frac{\Ell_+-\Ell_-}{2},\Ell_+\right).\]
  Therefore, we obtain
  \[\sum_{i=1}^n\inner{R(X,e_i)e_i}{X}=\phi^2\Ric(\Ell_+,\Ell_+).\]
  
  Next, we consider the second through fourth terms in the integrand function of \eqref{eq_second_variation_general}.
  Since $\Nabla_{e_i}X$ can be written as
  \[\Nabla_{e_i}X=(e_i\phi)\Ell_++\phi\Nabla_{e_i}\Ell_+\]
  for each unit vector field $e_i$, we can calculate
  \[\inner{\Nabla_{e_i}X}{\Ell_+}=\phi\inner{\Nabla_{e_i}\Ell_+}{\Ell_+}=\frac{1}{2}\phi e_i\inner{\Ell_+}{\Ell_+}=0.\]
  This implies that $(\Nabla_{e_i}X)^\bot$ and $\Ell_+$ are linearly dependent, therefore, we obtain \[\sum_{i=1}^n\inner{(\Nabla_{e_i}X)^\bot}{(\Nabla_{e_i}X)^\bot}=0.\]
  Moreover, by same calculations in \cite[Section 2]{Kishida24}, we obtain
  \begin{align*}
    &\sum_{i,j=1}^n\inner{\Nabla_{e_i}X}{e_j}\inner{\Nabla_{e_j}X}{e_i}=\phi^2\tr(A_+^2),\\
    &\sum_{i,j=1}^n\inner{\Nabla_{e_i}X}{e_i}\inner{\Nabla_{e_j}X}{e_j}=\phi^2\tr(A_+)^2=0.
  \end{align*}
  
  Finally, we have $\inner{X'}{\vbH}=0$ because of the definition of a characteristic variation.
  From these calculations, the proposition follows.
\end{proof}

\begin{remark}\label{rem_label}
  In the case of $n=2$, the second variational formula given in Proposition~\ref{prop_volume_var_nec} confirms the results given in \cite{AM10, HI15}.
  As noted in Andersson-Metzger~\cite{AM10}, for minimal hypersurfaces in a Riemannian manifold, or maximal hypersurfaces in a Lorentzian manifold, the second variation operator for volume is an elliptic operator of second order.
  In contrast, Proposition~\ref{prop_volume_var_nec} shows that the second variation operator with respect to a characteristic variation of a marginally trapped submanifold is an operator of order zero.
  This fact seems to be a remarkable phenomenon connecting to the volume maximality of a marginally trapped submanifold.
\end{remark}

By assuming that $M^{n+2}_1$ satisfies the null energy condition, we obtain the volume maximality for characteristic variations as follows:

\begin{proposition}\label{prop_second_variation}
  Let $\Sigma^n$ be a compact oriented manifold with boundary, and let $f:\Sigma^n\to M^{n+2}_1$ be a marginally trapped submanifold.
  Suppose that $M^{n+2}_1$ satisfies the null energy condition.
  Then, for any characteristic variation with fixed boundary, the second variation on the volume is non-positive.
\end{proposition}
\begin{proof}
  It is clear from \eqref{eq_volume_var_nec} and the definition of the null energy condition.
\end{proof}

\section{The null-space of a marginally trapped submanifold}\label{sec_3}

In this section, we give the definition of the ``null-space'' of a marginally trapped submanifold.
The definition here is a generalization of the null-space defined in \cite[Section 3]{Kishida24} for a space-like hypersurface in the light-cone.

\begin{definition}[A generalization of {\cite[Definition~3.1]{Kishida24}}]\label{def_null_space}
  Let $M^{n+2}_1$ be a light-like geodesically complete Lorentzian manifold and $f:\Sigma^n\to M^{n+2}_1$ a marginally trapped submanifold with respect to $\Ell_+$.
  Then, we define the map $\Phi_f:\R\times\Sigma^n\to M^{n+2}_1$ as
  \begin{equation}\label{eq_def_G}
    \Phi_f(t,x):=\exp_{f(x)}\bigl(t\Ell_+(x)\bigr)\qquad(t\in\R,\;x\in\Sigma^n),
  \end{equation}
  where $\exp$ is the exponential map of $M^{n+2}_1$.
  We set $\mathcal{N}_f^{n+1}:=\R\times\Sigma^n$, and consider that $\mathcal{N}_f^{n+1}$ is equipped with the induced metric determined by $\Phi_f$.
  This space $\mathcal{N}_f^{n+1}$ with the degenerate metric is called the \textit{null-space} with respect to $f$. 
\end{definition}

In the case of $M^{n+2}_1:=\R^{n+2}_1$, the null-space $\mathcal{N}_f^{n+1}$ can be considered as an ``L-complete null wave front'', which is defined in Akamine-Honda-Umehara-Yamada~\cite{AHUY25}.

\begin{proposition}[A generalization of {\cite[Proposition~3.5]{Kishida24}}]
  Let $f:\Sigma^n\to\R^{n+2}_1$ be a marginally trapped submanifold.
  Then, the map $\Phi_f$ defined as \eqref{eq_def_G} is an L-complete null wave front.
\end{proposition}
\begin{proof}
  Imitating the proof of \cite[Proposition~3.5]{Kishida24}, we can show that $\Phi_f$ is an L-complete null wave front.
\end{proof}

\[
\begin{tikzcd}
  \Sigma^n \arrow[r,"f"]\arrow[d,"\iota_f"] & M^{n+2}_1\\
  \mathcal{N}_f^{n+1}\arrow[ru,"\Phi_f"']
\end{tikzcd}
\]
\begin{center}
  Figure 1. The diagram related to $\mathcal{N}_f^{n+1}$.
\end{center}

By our construction, a marginally trapped submanifold $f$ is an embedded hypersurface in $\mathcal{N}_f^{n+1}$ by the isometric embedding $\iota_f:\Sigma^n\ni x\mapsto(0,x)\in\mathcal{N}_f^{n+1}$ (cf. Figure 1).
Therefore, a variation in $\mathcal{N}_f^{n+1}$ can be interpreted as a variation of $\iota_f$.
If the ambient space $M^{n+2}_1$ satisfies the null energy condition, Fact~\ref{fact_ahuy} implies that the null-space $\mathcal{N}_f^{n+1}$ (more precisely, $\Phi_f$) has singular points in general.
However, since $\Phi_f$ is an immersion on a tubular neighborhood of $\{0\}\times\Sigma^n$, singular points of the null-space may not be kept in mind when considering variations of $f$.

The following proposition states that for any variation in $\mathcal{N}_f^{n+1}$ can be regarded as a characteristic variation by using a coordinate transformation of the variation.

\begin{proposition}[A generalization of {\cite[Proposition~3.8]{Kishida24}}]\label{prop_exchange_variation}
  Let $\Sigma^n$ be a compact manifold with boundary, and $f:\Sigma^n\to M^{n+2}_1$ be a marginally trapped submanifold.
  Then, for any variation $G:(-\epsilon,\epsilon)\times\Sigma^n\to\mathcal{N}_f^{n+1}$ of $f$ in the null-space $\mathcal{N}_f^{n+1}$ with fixed boundary, there exist $\delta\in (0,\epsilon)$ and a characteristic variation $F:(-\delta,\delta)\times\Sigma^n\to M^{n+2}_1$ of $f$ with fixed boundary satisfying the following condition.
  \begin{itemize}
    \item The volume function $\textup{Vol}_F(t)$ of $F$ coincides with $\textup{Vol}_G(t)$, that is, the following equality holds:
    \[\textup{Vol}_F(t)=\int_{\Sigma^n}dV_t=\int_{\Sigma^n}d\widetilde{V}_t=\textup{Vol}_G(t)\qquad(t\in(-\delta,\delta)),\]
    where $dV_t$ and $d\widetilde{V}_t$ denote the volume forms of the metric induced by $F$ and $G$, respectively {\rm (cf. Figure 2)}.
  \end{itemize}
\end{proposition}
\[
  \begin{tikzcd}
    (-\delta,\delta)\times\Sigma^n\arrow[rd,"G"]\arrow[rr,bend left,"F"] & \Sigma^n \arrow{r}{f} \arrow[d,"\iota_f"] & M^{n+2}_1 \\
    & \mathcal{N}_f^{n+1} \arrow[ru,"\Phi_f"']
  \end{tikzcd}
\]
\begin{center}
  Figure 2. The diagram related to Proposition~\ref{prop_exchange_variation}.
\end{center}
\begin{proof}
  This proposition can be shown by replacing $\vbp$ and $\vbq$ in the proof of \cite[Proposition~3.8]{Kishida24} with $f$ and $\Ell_+$, respectively, and using the same discussion:
  
  We can write $G(t,x)=(\tau(t,x),\alpha(t,x))$ where $\tau:(-\epsilon,\epsilon)\times\Sigma^n\to\R$ and $\alpha:(-\epsilon,\epsilon)\times\Sigma^n\to\Sigma^n$ are two maps.
  By using exactly the same discussion in the proof of \cite[Proposition~3.8]{Kishida24},
  there exists $\delta\in(0,\epsilon)$ such that $\alpha(t,\cdot):\Sigma^n\to \Sigma^n$ is a diffeomorphism for each $t\in (-\delta,\delta)$.
  So we obtain a smooth map $\beta:(-\delta,\delta)\times \Sigma^n\to \Sigma^n$ which satisfies $\beta(t,\alpha(t,x))=x$.
  Here, we consider the variation
  \[F:(-\delta,\delta)\times\Sigma^n\ni(t,x)\mapsto\exp_{f(x)}\bigl(\tau(t,\beta(t,x))\Ell_+(x)\bigr)\in M^{n+2}_1.\]
  Clearly, $F$ is a variation of $f$ with fixed boundary.
  Since the variational vector field $X$ with respect to $F$ is written as
  \[X_x=\left(\left.\frac{\partial}{\partial t}\tau(t,\beta(t,x))\right|_{t=0}\right)\Ell_+(x)\qquad (x\in\Sigma^n),\]
  $F$ is an admissible variation.
  For each $x\in\Sigma^n$, the map $\R\ni t\mapsto F(t,x)\in M^{n+2}_1$ is a pregeodesic by the definition of $F$.
  This implies that $X'_x$ is parallel to $\Ell_+(x)$, that is, $F$ is a characteristic variation.
  
  We next show that the volume function $\Vol_F(t)$ of $F$ coincides with $\Vol_G(t)$ of $G$.
  We fix $t\in(-\delta,\delta)$ and set $F_t:=F(t,\cdot),\;G_t:=G(t,\cdot)$ and $\alpha_t:=\alpha(t,\cdot)$.
  Then we have \[F_t\circ\alpha_t=\Phi_f\circ G_t.\]
  Since $\alpha_t$ is a diffeomorphism and $\mathcal{N}_f^{n+1}$ is equipped with the metric induced by $\Phi_f$,
  the volume of the metric induced by $F_t$ is equal to that of the metric induced by $G_t$.
  Therefore, we have $\Vol_F(t)=\Vol_G(t)$.
\end{proof}

Finally we prove the theorem in Introduction.

\begin{proof}[Proof of Theorem in Introduction]
  This theorem follows by applying Propositions~\ref{prop_first_variation}, \ref{prop_second_variation} and \ref{prop_exchange_variation}.
\end{proof}

\begin{acknowledgements}
  The author would like to express my deep gratitude to Atsufumi Honda, Shunsuke Ichiki, Masaaki Umehara, and the reviewer for valuable comments and suggestions.
\end{acknowledgements}


\begin{thebibliography}{00}

  \bibitem{AHUY20}
  S. Akamine, A. Honda, M. Umehara, and K. Yamada,
  {\it Null hypersurfaces in Lorentzian manifolds with the null energy condition},
  J. Geom. Phys. {\bf 155} (2020), 103751, 6pp.
  
  \bibitem{AHUY25}
  S. Akamine, A. Honda, M. Umehara, and K. Yamada,
  {\it Null hypersurfaces as wave fronts in Lorentz-Minkowski space},
  J. Math. Soc. Japan. {\bf 77} (2025), 1--30.

  \bibitem{AM10}
  L. Andersson and J. Metzger,
  {\it Curvature estimates for stable marginally trapped surfaces},
  J. Differential Geom. {\bf 84} (2010), 231--265.

  \bibitem{HI15}
  A. Honda and S. Izumiya,
  {\it The light-like geometry of marginally trapped surfaces in Minkowski space-time},
  J. Geom. {\bf 106} (2015), 185--210.

  \bibitem{Kishida24}
  R. Kishida,
  {\it The volume of conformally flat manifolds as hypersurfaces in the light-cone},
  Differential Geom. Appl. {\bf 96} (2024), 102173.

  \bibitem{SW01}
  R. Schoen and J. Wolfson,
  {\it Minimizing area among Lagrangian surfaces: the mapping problem},
  J. Differential Geom. {\bf 58} (2001), 1--86.
\end{thebibliography}
\end{document}